\documentclass[a4paper,12pt]{amsart}
\title{Non-commutative width and Gopakumar-Vafa invariants}
\date{}
\author{Yukinobu Toda}

\usepackage{amscd}
\usepackage{amsmath}
\usepackage{amssymb}
\usepackage{amsthm}
\usepackage{float}
\usepackage[dvips]{graphicx}

\usepackage[all,ps,dvips]{xy}

\usepackage{array}
\usepackage{amscd}
\usepackage[all]{xy}

\DeclareFontFamily{U}{rsfs}{%
\skewchar\font127}
\DeclareFontShape{U}{rsfs}{m}{n}{%
<-6>rsfs5<6-8.5>rsfs7<8.5->rsfs10}{}
\DeclareSymbolFont{rsfs}{U}{rsfs}{m}{n}
\DeclareSymbolFontAlphabet
{\mathrsfs}{rsfs}
\DeclareRobustCommand*\rsfs{%
\@fontswitch\relax\mathrsfs}

\theoremstyle{plain}
\newtheorem{thm}{Theorem}[section]

\newtheorem{lem}[thm]{Lemma}

\newtheorem{defi}[thm]{Definition}
\newtheorem{rmk}[thm]{Remark}

\newtheorem{prop-defi}[thm]{Proposition-Definition}
\newtheorem{thm-defi}[thm]{Theorem-Definition}
\newtheorem{lem-defi}[thm]{Lemma-Definition}

\newtheorem{exam}[thm]{Example}

\newdimen\argwidth
\def\db[#1\db]{
 \setbox0=\hbox{$#1$}\argwidth=\wd0
 \setbox0=\hbox{$\left[\box0\right]$}
  \advance\argwidth by -\wd0
 \left[\kern.3\argwidth\box0 \kern.3\argwidth\right]}

\newcommand{\cC}{\mathcal{C}}

\newcommand{\eE}{\mathcal{E}}
\newcommand{\fF}{\mathcal{F}}

\newcommand{\hH}{\mathcal{H}}
\newcommand{\iI}{\mathcal{I}}

\newcommand{\lL}{\mathcal{L}}

\newcommand{\nN}{\mathcal{N}}
\newcommand{\oO}{\mathcal{O}}
\newcommand{\pP}{\mathcal{P}}
\newcommand{\qQ}{\mathcal{Q}}

\newcommand{\uU}{\mathcal{U}}

\newcommand{\xX}{\mathcal{X}}
\newcommand{\yY}{\mathcal{Y}}

\newcommand{\Supp}{\mathop{\rm Supp}\nolimits}
\newcommand{\Hom}{\mathop{\rm Hom}\nolimits}

\newcommand{\dotimes}{\stackrel{\textbf{L}}{\otimes}}
\newcommand{\dR}{\mathbf{R}}
\newcommand{\dL}{\mathbf{L}}

\newcommand{\Hilb}{\mathop{\rm Hilb}\nolimits}

\newcommand{\Ext}{\mathop{\rm Ext}\nolimits}
\newcommand{\Spec}{\mathop{\rm Spec}\nolimits}

\newcommand{\Coh}{\mathop{\rm Coh}\nolimits}

\newcommand{\cneq}{\mathrel{\raise.095ex\hbox{:}\mkern-4.2mu=}}
\newcommand{\eqcn}{\mathrel{=\mkern-4.5mu\raise.095ex\hbox{:}}}

\newcommand{\oPPer}{\mathop{\rm ^{0}Per}\nolimits}

\newcommand{\iPPer}{\mathop{\rm ^{-1}Per}\nolimits}

\newcommand{\pPPer}{\mathop{\rm ^{\mathit{p}}Per}\nolimits}

\newcommand{\modu}{\mathop{\rm mod}\nolimits}
\newcommand{\End}{\mathop{\rm End}\nolimits}

\begin{document}
\maketitle

\begin{abstract}
We show that the non-commutative widths
for flopping curves on smooth 3-folds introduced by 
Donovan-Wemyss are described by Katz's 
genus zero 
Gopakumar-Vafa invariants. 
\end{abstract}

\section{Introduction}
\subsection{Result}
Let $X$ be a smooth quasi-projective 
complex 3-fold and 
\begin{align*}
f \colon X \to Y
\end{align*}
a birational flopping contraction 
which contracts a single rational curve 
$\mathbb{P}^1 \cong C \subset X$ to a point 
$p\in Y$. 
In the paper~\cite{WM}, Donovan-Wemyss 
introduced a new invariant associated to $f$, 
the \textit{contraction algebra} $A_{\rm{con}}$, 
given by the universal non-commutative deformation 
algebra of the curve $C$ in $X$. 
The algebra $A_{\rm{con}}$ is finite dimensional, 
and it is commutative if and only if $C$ is not a $(1, -3)$-curve. 
Furthermore if $A_{\rm{con}}$ is commutative, 
the dimension of $A_{\rm{con}}$ coincides with Reid's width~\cite{Rei}
of $C$. Based on this observation, 
Donovan-Wemyss defined the following generalizations of 
Reid's width
\begin{align*}
\mathrm{wid}(C) \cneq \dim_{\mathbb{C}} A_{\rm{con}}, \quad 
\mathrm{cwid}(C) \cneq \dim_{\mathbb{C}} A_{\rm{con}}^{\rm{ab}}
\end{align*}
which they called \textit{non-commutative width} and \textit{commutative width}
respectively. 

On the other hand, Katz~\cite{Katz}
defined \textit{genus zero Gopakumar-Vafa (GV) invariants}
as virtual numbers of one 
dimensional stable sheaves on $X$. 
For $j\ge 1$, the genus zero GV invariant 
$n_{j}\in \mathbb{Z}_{\ge 0}$
 of curve class $j[C]$ on $X$ is shown 
in~\cite{Katz} to coincide
with the multiplicity of the Hilbert scheme
of $X$ at some subscheme $C^{(j)} \subset X$
with curve class $j[C]$. 
The purpose of this short note is 
to describe Donovan-Wemyss's widths in terms of 
Katz's genus zero GV invariants. 
The main result is as follows: 
\begin{thm}\label{thm:main}
We have the following formulas
\begin{align}
\mathrm{wid}(C) =\sum_{j=1}^{l} j^2 \cdot n_j, \quad
\mathrm{cwid}(C) =n_1. 
\end{align}
Here $l$ is the scheme theoretic length of $f^{-1}(p)$ at $C$. 
\end{thm}
Here we remark that the identity of $\mathrm{cwid}(C)$ is almost 
obvious from the definitions, and the identity of $\mathrm{wid}(C)$
is more interesting. 
The result of Theorem~\ref{thm:main} indicates that one 
can study non-commutative widths
without using non-commutative algebras\footnote{Wemyss pointed out 
to the author that the non-commutative widths are
commutative things, as they are 
computed using some $\Ext$-groups on 
commutative algebras. See~\cite[Remark~5.2]{WM}.  
}. 
Conversely, one may compute genus zero GV invariants
by computing contraction algebras. 
The proof of Theorem~\ref{thm:main}
is an easy application of the main result of~\cite{WM}, combined with 
some deformation argument.
By~\cite{WM}, the algebra $A_{\rm{con}}$ defines the non-commutative
twist functor, describing 
Bridgeland-Chen's flop-flop autoequivalence
of $D^b \Coh(X)$.
On the other hand, after taking the completion at $p$, the 
morphism $f$ deforms to flopping contractions of disjoint $(-1, -1)$-curves, 
such that the number 
of $(-1, -1)$-curves with curve
class $j[C]$ coincides with $n_j$.
Now the flop-flop autoequivalence deforms
along the deformation of $f$, hence
the non-commutative twist functor
also deforms:
the resulting deformation 
is 
a composition of 
Seidel-Thomas's spherical twists along $(-1, -1)$-curves. 
 We then relate the
Hilbert polynomial of 
a cohomology sheaf of 
the kernel object of 
the non-commutative twist functor 
with that of the above
composition of the spherical twists, 
and obtain the desired identity of $\mathrm{wid}(C)$. 
\subsection{Examples and a Remark}
Here we describe some examples of Theorem~\ref{thm:main}. 
\begin{exam}
In Theorem~\ref{thm:main}, we have 
$l=1$ if and only if 
$C$ is either a $(-1, -1)$ or a $(0, -2)$-curve. 
In this case, 
we have $\mathrm{wid}(C)=\mathrm{cwid}(C)$, 
and it coincides with Reid's width (cf.~\cite[Example~3.12]{WM}). 
On the other hand, the genus zero GV invariant $n_1$ also 
coincides with Reid's width as indicated in~\cite[Section~1]{BKL}. 
\end{exam}

\begin{exam}
Suppose that 
$Y=\Spec R_k$, where $R_k$ is defined by 
\begin{align*}
R_k=\mathbb{C}[u, v, x, y]/(u^2+v^2 y=x(x^2 + y^{2k+1})). 
\end{align*}
There is a flopping contraction $f \colon X \to Y$ 
with $l=2$. The contraction algebra $A_{\rm{con}}$ is computed
 in~\cite[Example~3.14]{WM}
\begin{align*}
&A_{\rm{con}} \cong \mathbb{C} \langle
x, y \rangle /(xy=-yx, x^2=y^{2k+1})\\
&A_{\rm{con}}^{\rm{ab}} \cong 
\mathbb{C}[x, y]/(xy=0, x^2=y^{2k+1}).  
\end{align*}
It follows that 
\begin{align*}
\mathrm{wid}(C)=3(2k+1), \quad 
\mathrm{cwid}(C)=2k+3. 
\end{align*}
The result of Theorem~\ref{thm:main}
indicates that $n_1=2k+3$ and $n_2=k$. 
\end{exam}
We also have the following remark: 
\begin{rmk}
We have $n_j \ge 1$ for $1\le j\le l$. 
So Theorem~\ref{thm:main} implies
that
\begin{align*}
\mathrm{wid}(C) \ge \sum_{j=1}^{l} j^2. 
\end{align*}
The above lower bound 
is better than the lower 
bound in~\cite[Remark~3.17]{WM}. 
\end{rmk}

\subsection{Acknowledgment}
I would like to thank Michael Wemyss and 
Will Donovan
for 
valuable comments on the manuscript. I would also like to thank Tom Bridgeland 
for checking a misprint of his paper~\cite{Br1}, 
and allowing me to correct it in Appendix~B. 
This work is supported by World Premier 
International Research Center Initiative
(WPI initiative), MEXT, Japan. This work is also supported by Grant-in Aid
for Scientific Research grant (No.~26287002)
from the Ministry of Education, Culture,
Sports, Science and Technology, Japan.

\section{Preliminary}
\subsection{3-fold flopping contractions}\label{subsec:flop}
Let $X$ be a smooth quasi-projective complex 3-fold. 
By definition, a \textit{flopping contraction} is a 
birational morphism 
\begin{align}\label{fXY}
f \colon X \to Y
\end{align}
which is isomorphic in codimension one, 
$Y$ has only Gorenstein singularities and 
the relative Picard number of $f$ equals to one. 
In what follows, we always assume that 
the exceptional locus $C$ of $f$ is
isomorphic to $\mathbb{P}^1$, and set
\begin{align*}
p \cneq f(C) \in Y. 
\end{align*}
We say that $C \subset X$ is $(a, b)$ curve if 
$N_{C/X}$ is isomorphic to $\oO_C(a) \oplus \oO_C(b)$. 
It is well-known that 
$(a, b)$
 is either one of the following: 
\begin{align*}
(a, b)=(-1, -1), \ (0, -2), \ (1, -3). 
\end{align*}
We denote by $l$ the length of $\oO_{f^{-1}(p)}$ at the 
generic point of $C$, where 
$f^{-1}(p)$ is the scheme theoretic fiber of 
$f$ at $p$. 
Then we have 
\begin{align*}
l \in \{1, 2, 3, 4, 5, 6\}
\end{align*}
and 
$l=1$ if and only if $C$ is not a $(1, -3)$-curve
(cf.~\cite[Section~1]{KaMo}). 
Moreover if $l=1$, then
we have 
\begin{align}\label{nwidth}
\widehat{\oO}_{Y, p} \cong \mathbb{C}\db[x, y, z, w \db]/(x^2+y^2+z^2+w^{2k}) 
\end{align}
for some $k\in \mathbb{Z}_{\ge 1}$. 
The number $k$ is called \textit{width} of $C$ 
in~\cite{Rei}.

\subsection{Contraction algebras}\label{subsec:cont}
In the setting of Subsection~\ref{subsec:flop}, 
we set $R =\widehat{\oO}_{Y, p}$, 
and take the following 
completion of (\ref{fXY})
\begin{align}\label{complete}
\widehat{f} \colon 
\widehat{X} \cneq X \times_Y \Spec R
\to \widehat{Y} \cneq \Spec R. 
\end{align}
Then there is a line bundle $\lL$ on $\widehat{X}$ such that 
$\deg(\lL|_{C})=1$. 
We define the vector bundle $\nN$ on $\widehat{X}$ 
to be the extension
\begin{align*}
0 \to \lL^{-1} \to \nN \to \oO_{\widehat{X}}^{\oplus r} \to 0
\end{align*}
given by the minimum generators of 
$H^1(\widehat{X}, \lL^{-1})$. 
We set 
$\uU \cneq \oO_{\widehat{X}} \oplus \nN$, 
$N \cneq \widehat{f}_{\ast}\nN$ and 
\begin{align*}
A \cneq \End_{\widehat{X}}(\uU) 
\cong \End_{R}(R
\oplus N). 
\end{align*}
By Van den Bergh~\cite[Section~3.2.8]{MVB}, 
we have a derived equivalence
\begin{align}\label{deq}
\dR \Hom_{\widehat{X}}(\uU, -) \colon 
D^b \Coh(\widehat{X}) \stackrel{\sim}{\to} D^b \modu A
\end{align}
whose inverse is given by $- \dotimes_{A} \uU$. 
Here $\modu A$ is the category of finitely generated right 
$A$-modules. 
\begin{defi}\emph{(\cite[Definition~2.11]{WM})}
The contraction algebra $A_{\rm{con}}$ is defined to be
$A/I_{\rm{con}}$, where $I_{\rm{con}}$ is the two 
sided ideal of $A$ consisting of morphisms 
$R\oplus N \to R
\oplus N$
factoring through a member of $\mathrm{add}(R)$. 
Here $\mathrm{add}(R)$ is the set of summands of finite sums of $R$. 
\end{defi}
By~\cite[Proposition~2.12]{WM},
the algebra $A_{\rm{con}}$ is finite dimensional.  
\begin{rmk}
The algebra $A_{\rm{con}}$ is commutative if and only if 
$C$ is not a $(1, -3)$-curve (cf.~\cite[Theorem~3.15]{WM}). In this case, 
$A_{\rm{con}}$ is isomorphic to 
$\mathbb{C}[t]/(t^k)$, where $k$ is the 
width of $C$ which appears in (\ref{nwidth}). 
See~\cite[Example~3.12]{WM}. 
\end{rmk}
The contraction algebra $A_{\rm{con}}$
coincides with
the universal 
algebra 
which represents the 
non-commutative deformation 
functor of $\oO_C(-1)$
\begin{align}\label{ADef}
\mathrm{Def}_{\oO_{C}(-1)} \colon 
\mathrm{Art}_1 \to \mathrm{Sets}.
\end{align}
Here $\mathrm{Art}_1$ is the category 
of finite dimensional $\mathbb{C}$-algebras $\Gamma$ with some 
additional conditions, 
and the functor (\ref{ADef})
assigns each $\Gamma$
to the set of isomorphism classes
of flat deformation of $\oO_C(-1)$
to $\Coh(\oO_X \otimes_{\mathbb{C}} \Gamma)$. 
We refer~\cite[Section~2]{WM} for details of the 
functor (\ref{ADef}).  
Since $\mathrm{A}_{\rm{con}}$ represents the functor 
(\ref{ADef}), 
there is the universal 
non-commutative deformation of $\oO_C(-1)$
\begin{align}\label{univ:E}
\eE \in \Coh(\oO_X \otimes_{\mathbb{C}}A_{\rm{con}}).
\end{align}
Let $A_{\rm{con}}^{\rm{ab}}$ 
be the abelization of $A_{\rm{con}}$. 
The algebra $A_{\rm{con}}^{\rm{ab}}$ is 
a commutative Artinian local $\mathbb{C}$-algebra, which 
represents the commutative deformation functor
\begin{align}\label{cDef}
\mathrm{cDef}_{\oO_{C}(-1)} \colon 
\mathrm{cArt}_1 \to \mathrm{Sets}.
\end{align}
Here $\mathrm{cArt}_1$ is 
the category of commutative Artinian local 
$\mathbb{C}$-algebras, 
and the functor 
(\ref{cDef}) is the restriction of
the functor (\ref{ADef}) to $\mathrm{cArt}_1$. 
We refer~\cite[Section~3]{WM}
for details of the above representabilities.

\subsection{Flop equivalences}
The contraction algebra $A_{\rm{con}}$ plays an important role 
in describing Bridgeland-Chen's flop-flop autoequivalence. 
Let us consider the flop diagram of (\ref{fXY})
\begin{align}\label{dia:flop}
\xymatrix{
X \ar[dr]_{f}  \ar@{.>}[rr]^{\phi} &  & X^{\dag}
 \ar[dl]^{f^{\dag}} \\
& Y.  &
}
\end{align}
By~\cite{Br1} and~\cite{Ch}, we have the derived equivalence
\begin{align}\label{Feq}
\Phi_{X \to X^{\dag}}^{\oO_{X \times_Y X^{\dag}}} \colon 
D^b \Coh(X) \stackrel{\sim}{\to} D^b \Coh(X^{\dag}). 
\end{align}
Here we use the notation in Appendix~A for the 
Fourier-Mukai functors. 
Composing (\ref{Feq}) twice, we obtain the autoequivalence
\begin{align}\label{FF:eq}
\Phi_{X^{\dag} \to X}^{\oO_{X \times_Y X^{\dag}}}
\circ \Phi_{X \to X^{\dag}}^{\oO_{X \times_Y X^{\dag}}}
\colon D^b \Coh(X) \stackrel{\sim}{\to} D^b \Coh(X). 
\end{align}
The result of~\cite[Proposition~7.18]{WM} shows that 
(\ref{FF:eq}) has an inverse
isomorphic to the non-commutative twist functor $T_{\eE}$
associated to the universal object (\ref{univ:E}).  
Namely $T_{\eE}$ is the autoequivalence of $D^b \Coh(X)$
which fits into the distinguished triangle
\begin{align}\label{nc:twist}
\dR \Hom(\eE, F) \dotimes_{A_{\rm{con}}} \eE \to F \to T_{\eE}(F)
\end{align}
for any $F \in D^b \Coh(X)$. 
If $C$ is a $(-1, -1)$-curve, the 
functor $T_{\eE}$
coincides with 
Seidel-Thomas twist~\cite{ST}
along $\oO_C(-1)$.
If $C$ is a $(0, -2)$-curve, 
then $T_{\eE}$ coincides with 
the author's generalized
twist~\cite{Tgen}\footnote{In~\cite{Tgen}, it was stated 
that $T_{\eE}$ is isomorphic to (\ref{FF:eq}), 
but it was wrong:
the correct statement is $T_{\eE}$ is an inverse of (\ref{FF:eq}). 
We explain details in Appendix~B.}. 
The kernel object of the
equivalence $T_{\eE}$ is given by 
 \begin{align*}
\mathrm{Cone} \left( \dR \Hom_{A}(A_{\rm{con}}, A) \dotimes_{A^{\rm{op}} 
\otimes A} (\uU^{\vee} \boxtimes \uU) \to \oO_{\Delta_{X}}   \right). 
\end{align*}
Here $\Delta_{X} \subset X \times X$ is the 
diagonal~(cf.~\cite[Lemma~6.16]{WM}). 
\begin{lem}\label{lem:F}
The object $\dR \Hom_{A}(A_{\rm{con}}, A) \dotimes_{A^{\rm{op}} 
\otimes A} (\uU^{\vee} \boxtimes \uU)$
is isomorphic to $\fF[-2]$ for 
$\fF \in \Coh(X \times X)$
satisfying the following: 
there is a filtration
\begin{align*}
0=\fF_0 \subset \fF_1 \subset \cdots \subset \fF_{\dim A_{\rm{con}}} =\fF
\end{align*}
such that each subquotient $\fF_j/\fF_{j-1}$ is isomorphic to 
$\oO_C(-1) \boxtimes \oO_C(-1)$. 
\end{lem}
\begin{proof}
By the definition of $\mathrm{Art}_1$ in~\cite[Definition~2.1]{WM},
there 
is a $\mathbb{C}$-algebra homomorphism
$A_{\rm{con}} \to \mathbb{C}$ 
such that its kernel $\mathfrak{n} \subset A_{\rm{con}}$ is nilpotent. 
The ideal $\mathfrak{n} \subset A_{\rm{con}}$ is 
two-sided, and $A_{\rm{con}}/\mathfrak{n}$ is 
a one dimensional $A^{\rm{op}} \otimes A$-module. 
We have the filtration of $A^{\rm{op}} \otimes A$-modules
\begin{align*}
0=\mathfrak{n}^m \subset \mathfrak{n}^{m-1} \subset \cdots \subset
\mathfrak{n} \subset A_{\rm{con}}
\end{align*}
for some $m>0$ 
such that each subquotient
$\mathfrak{n}^i/\mathfrak{n}^{i+1}$
is an $A_{\rm{con}}/\mathfrak{n}$-module. 
Since $A_{\rm{con}}/\mathfrak{n}=\mathbb{C}$, 
the object $\mathfrak{n}^i/\mathfrak{n}^{i+1}$
is a finite direct sum of 
$A_{\rm{con}}/\mathfrak{n}$. 
Therefore it is enough to show that 
\begin{align}\label{enough}
&\dR \Hom_{A}(A_{\rm{con}}/\mathfrak{n}, A) \dotimes_{A^{\rm{op}} 
\otimes A} (\uU^{\vee} \boxtimes \uU) \\
\notag
&\cong \oO_C(-1) \boxtimes \oO_C(-1)[-2]. 
\end{align}
Let $S \in \modu A$ be the object given by 
\begin{align*}
S \cneq \dR \Hom_{\widehat{X}}(\uU, \oO_C(-1)). 
\end{align*}
Note that we have $\dim_{\mathbb{C}}S=1$.
The object $S$ is the unique simple $A_{\rm{con}}$-module
(cf.~\cite[Section~2.3]{WM}), 
hence $A_{\rm{con}}/\mathfrak{n}$ viewed
as a right $A_{\rm{con}}$-module is isomorphic to $S$. 
On the other hand, the vector bundle 
$\uU^{\vee} \boxtimes \uU$ on 
$\widehat{X} \times \widehat{X}$
is a tilting vector bundle.
Hence we have a derived equivalence
\begin{align*}
\dR \Hom_{\widehat{X} \times \widehat{X}}(\uU^{\vee} \boxtimes \uU, -) \colon 
D^b \Coh(\widehat{X} \times \widehat{X})
\stackrel{\sim}{\to} D^b \modu(A^{\rm{op}} \otimes A)
\end{align*}
with inverse given by 
$- \dotimes_{A^{\rm{op}} \otimes A} (\uU^{\vee} \boxtimes \uU)$. 
Let $\mathbb{D}$ be the dualizing functor
$\dR \hH om_{\widehat{X}}(-, \oO_{\widehat{X}})$ on $D^b \Coh(\widehat{X})$. 
We have
$\mathbb{D}(\oO_C(-1)) \cong \oO_C(-1)[-2]$, and 
\begin{align*}
&\dR \Hom_{\widehat{X} \times \widehat{X}}
(\uU^{\vee} \boxtimes \uU, \oO_C(-1) \boxtimes \oO_C(-1)[-2]) \\
&\cong \dR \Hom_{\widehat{X}\times \widehat{X}}(\mathbb{D}(\uU)
 \boxtimes \uU, \mathbb{D}(\oO_C(-1)) \boxtimes \oO_C(-1)) \\
&\cong \dR \Hom_{\widehat{X}}(\oO_C(-1), \uU)
 \otimes_{\mathbb{C}} \dR \Hom_{\widehat{X}}(\uU, \oO_C(-1)) \\
&\cong \dR \Hom_A(S, A) \otimes_{\mathbb{C}} S \\
&\cong \dR \Hom_{A}(A_{\rm{con}}/\mathfrak{n}, A).
\end{align*}
Therefore we obtain the desired isomorphism (\ref{enough}). 
\end{proof}

\subsection{Genus zero Gopakumar-Vafa invariants}
The genus zero GV invariants defined in~\cite{Katz}
count one dimensional stable sheaves $F$
on Calabi-Yau 3-folds satisfying $\chi(F)=1$. 
In the setting of Subsection~\ref{subsec:flop}, the 
variety $X$ may not be Calabi-Yau, 
but so in a neighborhood of $C$. 
Since $C$ is rigid in $X$, we can 
define the genus zero GV invariant with curve 
class $j[C]$ on $X$
as well. 
Indeed in~\cite{Katz}, the genus zero GV invariants
of $X$ are shown to coincide with 
the multiplicities of the Hilbert scheme of $X$
at some subschemes supported on $C$. 
Let $p\in H \subset Y$ be a general hypersurface,  
and $\overline{H} \subset X$ its proper transform. 
Then we have $C \subset \overline{H}$. 
Let $I \subset \oO_{\overline{H}}$ be
the ideal sheaf of $C$. 
For $j\ge 1$, we have the subscheme
$C^{(j)} \subset X$ 
given by
\begin{align*}
\oO_{C^{(j)}}=
(\oO_{\overline{H}}/I^j)/Q
\end{align*}
 where 
$Q$ is the maximum zero dimensional subsheaf of 
$\oO_{\overline{H}}/I^j$. 
Let $\Hilb(X)$ be the Hilbert scheme 
parameterizing closed subschemes in 
$X$. 
If $1\le j\le l$, it is shown 
in~\cite[Section~2.1]{BKL}, 
that 
$C^{(j)}$ is the isolated point in $\Hilb(X)$, and 
the following number is defined: 
\begin{defi}\label{def:nj}
For $1\le j\le l$, we define $n_j \in \mathbb{Z}_{\ge 1}$
to be
\begin{align*}
n_j \cneq \dim_{\mathbb{C}} \oO_{\Hilb(X), C^{(j)}}. 
\end{align*}
By convention, we define $n_j=0$ for $j>l$. 
\end{defi}
Since $\oO_{\Hilb(X), C^{(j)}}$ is 
a finitely generated Artinian $\mathbb{C}$-algebra, 
the number $n_j$
is well-defined. 
If $l=1$, the number $n_1$ equals to 
the width $k$ in (\ref{nwidth})
as indicated in~\cite[Section~1]{BKL}. 
In general, Katz~\cite{Katz} shows that 
$n_j$ coincides with the 
genus zero GV invariant of $X$ 
with curve class $j[C]$.  

The number $n_j$ also appears
in the context of deformations in the following way. 
By~\cite[Section~2.1]{BKL},
there exists a flat deformation 
of (\ref{complete})
\begin{align}\label{deform}
\xymatrix{ 
\xX \ar[r]^{g}\ar[dr] & \yY \ar[d] \\
&  T
}
\end{align} 
where $T$ is 
a Zariski open neighborhood of 
$0\in \mathbb{A}^1$
such that $g_0 \colon \xX_0 \to \yY_0$ 
is isomorphic to $\widehat{f}$
in (\ref{complete}), 
and $g_t \colon \xX_t \to \yY_t$
for $t \in T \setminus \{0\}$ 
is a flopping contraction whose 
exceptional locus is a disjoint union of
$(-1, -1)$-curves. 
Here $\xX_t, \yY_t$ are the fibers of 
$\xX \to T$, $\yY \to T$ at $t \in T$
respectively. 
Then the number $n_j$ coincides with 
the number of $g_t$-exceptional 
$(-1, -1)$-curves $C' \subset \xX_t$
for $t\neq 0$ whose curve class equals to $j[C]$, i.e. 
for any line bundle $\lL$ on $\xX$, we have
\begin{align}\label{deg}
\deg(\lL|_{C'})=j \deg(\lL|_{C})
\end{align}
 where 
we regard $C$ as a curve on the central fiber of 
$\xX \to T$. 
In what follows, we write 
the exceptional locus of $g_t$ for $t\neq 0$ 
as
\begin{align*}
C_{j, k} \subset \xX_t, \ 
1\le j\le l, \ 1\le k\le n_j
\end{align*}
where $C_{j, k}$ is a $(-1, -1)$-curve with curve class $j[C]$.

\section{Proof of Theorem~\ref{thm:main}}
\begin{proof}
The identity $\mathrm{cwid}(C)=n_1$
is almost obvious from the definitions of both sides. 
Indeed since $A_{\rm{con}}^{\rm{ab}}$
represents the commutative deformation functor
(\ref{cDef}), the scheme
$\Spec A_{\rm{con}}^{\rm{ab}}$ is the component of the moduli 
scheme of one dimensional 
stable sheaves on $X$ containing $\oO_{C}(-1)$. 
By tensoring the line bundle $\lL$ in Subsection~\ref{subsec:cont}, 
the scheme $\Spec A_{\rm{con}}^{\rm{ab}}$ is 
isomorphic to the component of the 
moduli scheme of stable sheaves
containing $\oO_C$, which defines the invariant $n_1$. 
By the proof of~\cite[Proposition~3.3]{Katz}, the degree of the virtual 
fundamental cycle of $\Spec A_{\rm{con}}^{\rm{ab}}$ 
coincides with the dimension of $A_{\rm{con}}^{\rm{ab}}$.
Therefore $\mathrm{cwid}(C)=n_1$ holds. 

We show the identity of $\mathrm{wid}(C)$. 
The morphism $g$ in (\ref{deform}) is a flopping contraction, 
and the argument of~\cite[Section~6]{Ch} shows that 
$g$ admits a flop  
 \begin{align*}
\xymatrix{
\xX \ar[dr]_{g}  \ar@{.>}[rr]^{\psi} &  & \xX^{\dag}
 \ar[dl]^{g^{\dag}} \\
& \yY  &
}
\end{align*}
such that we have the derived equivalence
\begin{align*}
\Phi_{\xX \to \xX^{\dag}}^{\oO_{\xX \times_{\yY} \xX^{\dag}}}
\colon D^b \Coh(\xX) \stackrel{\sim}{\to} D^b \Coh(\xX^{\dag}). 
\end{align*}
By composing the above equivalence 
twice, we obtain the autoequivalence
\begin{align}\label{Xtwice}
\Phi_{\xX^{\dag} \to \xX}^{\oO_{\xX \times_{\yY} \xX^{\dag}}} \circ
\Phi_{\xX \to \xX^{\dag}}^{\oO_{\xX \times_{\yY} \xX^{\dag}}}
\colon D^b \Coh(\xX) \stackrel{\sim}{\to} D^b \Coh(\xX). 
\end{align}
Let $\Psi$ be an inverse of the equivalence (\ref{Xtwice}), 
and
\begin{align*}
\pP \in D^b \Coh(\xX \times_{T} \xX)
\end{align*}
the kernel object of $\Psi$. 
By~\cite[Lemma~6.1]{Ch}, for each 
$t \in T$, we have the commutative diagram
\begin{align*}
\xymatrix{
D^b \Coh(\xX) \ar[r]^{\Psi} \ar[d]^{\dL i_t^{\ast}}
 & D^b \Coh(\xX) \ar[d]^{\dL i_t^{\ast}} \\
D^b \Coh(\xX_t) \ar[r]^{\Psi_t} & D^b \Coh(\xX_t). 
}
\end{align*}
Here $i_t \colon \xX_t \hookrightarrow \xX$ is the inclusion, and 
$\Psi_t$ is
the Fourier-Mukai functor with kernel 
$\pP_t \cneq \dL j_t^{\ast} \pP$, where 
$j_t$ is the inclusion
\begin{align*}
j_t \cneq (i_t \times i_t) \colon \xX_t \times \xX_t \hookrightarrow \xX \times_{T}\xX. 
\end{align*}
The functor $\Psi_t$ is an equivalence, 
and it has an inverse given by the composition
(cf.~\cite[Corollary~4.5]{Ch})
\begin{align}\label{Xtwice2}
\Phi_{\xX^{\dag}_t \to \xX_t}^{\oO_{\xX_t \times_{\yY_t} \xX^{\dag}_t}} \circ
\Phi_{\xX_t \to \xX^{\dag}_t}^{\oO_{\xX_t \times_{\yY_t} \xX^{\dag}_t}}
\colon D^b \Coh(\xX_t) \stackrel{\sim}{\to} D^b \Coh(\xX_t). 
\end{align}
Therefore by~\cite[Proposition~7.18]{WM},
the equivalence $\Psi_0$ is isomorphic to the 
non-commutative twist functor $T_{\eE}$ in (\ref{nc:twist}).
By the uniqueness of Fourier-Mukai kernels in 
Lemma~\ref{lem:unique} below, 
we have
\begin{align}\label{isom:0}
\pP_0 \cong \mathrm{Cone} \left( \fF_0[-2] \to \oO_{\Delta_{\xX_0}}   \right).
\end{align}
Here $\fF_0$ is a sheaf $\fF$
on $X \times X$ given in Lemma~\ref{lem:F}, 
restricted to $\widehat{X} \times \widehat{X}$.  

For $t\neq 0$, the 
birational map 
$\xX_t \dashrightarrow \xX_t^{\dag}$
is the composition of flops
at $(-1, -1)$-curves
$C_{j, k}$ for $1\le j\le l$, $1\le k\le n_j$. 
Hence the equivalence 
$\Psi_t$ for $t\neq 0$ is isomorphic to the 
compositions of all the 
spherical twists along 
$\oO_{C_{j, k}}(-1)$
for $1\le j\le l$, 
$1\le k\le n_j$. 
Therefore using Lemma~\ref{lem:unique} again, we have
\begin{align}\label{isom:t}
\pP_t \cong \mathrm{Cone} \left( 
\fF_t[-2] \to
 \oO_{\Delta_{\xX_t}}  \right)
\end{align}
where $\fF_t$ is a sheaf on $\xX_t \times \xX_t$ defined by
\begin{align}\label{def:Ft}
\fF_t \cneq 
\bigoplus_{j=1}^{l}
\bigoplus_{k=1}^{n_j}
\oO_{C_{j, k}}(-1) \boxtimes 
\oO_{C_{j, k}}(-1). 
\end{align}
\begin{lem}\label{lem:Hi}
We have $\hH^i(\pP)=0$ for $i\neq 0, 1$. 
\end{lem}
\begin{proof}
For any $t\in T$, 
we have the distinguished triangle
\begin{align*}
\pP \to \pP \to j_{t\ast} \pP_t.
\end{align*}
By (\ref{isom:0}) and (\ref{isom:t}), we have
$\hH^i(\pP_t)=0$ for any $t\in T$
and $i\neq 0, 1$. 
By taking the long exact sequence of cohomologies of the above 
triangle, 
we obtain $j_t^{\ast} \hH^i(\pP)=0$ for any 
$t \in T$ and $i\neq 0, 1$. 
Therefore we have $\hH^i(\pP)=0$ for $i\neq 0, 1$. 
\end{proof}

\begin{lem}
We have $\hH^0(\pP) \cong \oO_{\Delta_{\xX}}$ and 
$\hH^1(\pP)$ is flat over $T$. 
Furthermore we have 
$j_t^{\ast}\hH^1(\pP) \cong \fF_t$
for any $t\in T$. 
\end{lem}
\begin{proof}
By Lemma~\ref{lem:Hi}, we have the distinguished triangle
in $D^b \Coh(\xX \times_{T} \xX)$
\begin{align*}
\hH^0(\pP) \to \pP \to \hH^1(\pP)[-1]. 
\end{align*}
Applying $\dL j_t^{\ast}$, we obtain the 
distinguished triangle in $D^b \Coh(\xX_t \times \xX_t)$
\begin{align*}
\dL j_t^{\ast}\hH^0(\pP) \to \pP_t \to \dL j_t^{\ast}\hH^1(\pP)[-1]. 
\end{align*}
By taking the long exact sequence of cohomologies, we have 
\begin{align*}
\dL j_t^{\ast}\hH^0(\pP) \cong j_t^{\ast}\hH^0(\pP), \ 
\fF_t \cong j_t^{\ast}\hH^1(\pP)
\end{align*}
and the exact sequence
\begin{align}\label{ex:long}
0 \to j_t^{\ast}\hH^0(\pP) \to \oO_{\Delta_{\xX_t}} \to 
\hH^{-1}(\dL j_t^{\ast}\hH^1(\pP)) \to 0. 
\end{align}

Below we denote by 
$\Delta \colon \xX \to \xX \times_T \xX$, 
$\Delta_t \colon \xX_t \to \xX_t \times \xX_t$
the diagonal morphisms, and 
$\cC$ the exceptional locus of $g \colon \xX \to \yY$. 
The isomorphism $\fF_t \cong j_t^{\ast}\hH^1(\pP)$
implies that $\hH^1(\pP)$ is supported on $\cC \times \cC$, 
hence $\hH^{-1}(\dL j_t^{\ast}\hH^1(\pP))$ is supported on 
$\cC_t \times \cC_t$. 
The exact sequence (\ref{ex:long}) also implies that 
$\hH^{-1}(\dL j_t^{\ast}\hH^1(\pP))$
is supported on $\Delta_{\xX_t}$, hence on 
$\Delta_{\xX_t} \cap (\cC_t \times \cC_t) =\Delta_t(\cC_t)$. 
It follows that 
$\hH^0(\pP)$ is 
written as $\Delta_{\ast}\iI$ for 
a rank one torsion free sheaf $\iI$ on $\xX$, and
the exact sequence (\ref{ex:long})
is given by $\Delta_{t\ast}$ of the 
exact sequence of the following form
\begin{align}\label{form:}
0 \to i_t^{\ast}\iI \to \oO_{\xX_t} \to \oO_{\cC_t'} \to 0
\end{align}
for some subscheme $\cC_t' \subset \xX_t$ supported on $\cC_t$. 
Also by the generic flatness, 
there is a non-empty Zariski open subset $U \subset T$
such that  
$\hH^{-1}(\dL j_t^{\ast}\hH^1(\pP))=0$
for all $t\in U$. 
This implies that 
$\cC_t'=\emptyset$ for all $t\in U$,
hence $\iI$ is isomorphic to $\oO_{\xX}$ 
away from $\cC_t$ for $t \in T \setminus U$. 
By taking the double dual of $\iI$, we obtain the exact sequence
\begin{align*}
0 \to \iI \to \oO_{\xX} \to \oO_{\cC'} \to 0
\end{align*}
where $\cC'$ is supported on $\cC_t$ for $t\in T \setminus U$. 
If $\cC' \neq \emptyset$, then 
$j_t^{\ast}\hH^0(\pP) \cong \Delta_{t \ast} i_t^{\ast}\iI$ contains 
the non-zero sheaf $\Delta_{t\ast}\hH^{-1} \dL i_t^{\ast}(\oO_{\cC'})$
for some $t \in T \setminus U$ supported on $\cC_t$, 
which contradicts to (\ref{ex:long}). 
Therefore $\cC'=\emptyset$ and 
$\hH^0(\pP) \cong \oO_{\Delta_{\xX}}$ holds. 

Now in the sequence (\ref{form:}), we have 
$i_t^{\ast}\iI\cong \oO_{\xX_t}$ for any 
$t\in T$, hence $\cC_t'=\emptyset$
as $\cC_t'$ has codimension bigger than or equal to two in $\xX_t$. 
This implies that $\hH^{-1}(\dL j_t^{\ast}\hH^1(\pP))=0$
for any $t\in T$, hence $\hH^1(\pP)$ is flat over $T$. 
\end{proof}
By the above lemma, the sheaf $\fF_t$ for $t\neq 0$ is a flat deformation of 
$\fF_0$. Since they have compact supports, 
 $\fF_0$ and $\fF_t$ have the same Hilbert polynomials. 
It follows that, for a 
$g$-ample line bundle $\lL$ on $\xX$
with $d \cneq \deg(\lL|_{C})>0$, we have the equality
\begin{align}\label{eq:chi}
\chi(\fF_0 \otimes (\lL \boxtimes \lL))=
\chi(\fF_t \otimes (\lL \boxtimes \lL)). 
\end{align}
By Lemma~\ref{lem:F} and the Riemann-Roch theorem, we have 
\begin{align*}
\chi(\fF_0 \otimes (\lL \boxtimes \lL)) &=\dim_{\mathbb{C}}A_{\rm{con}} \cdot 
\chi(\oO_C(-1) \otimes \lL)^2 \\
&=\dim_{\mathbb{C}}A_{\rm{con}} \cdot d^2.  
\end{align*}
By the definition of $\fF_t$ for $t\neq 0$ in (\ref{def:Ft}),
we have
\begin{align*}
\chi(\fF_t \otimes (\lL \boxtimes \lL)) &=
\sum_{j=1}^{l} \sum_{k=1}^{n_j}
\chi(\oO_{C_{j, k}}(-1) \otimes \lL)^2 \\
&=\sum_{j=1}^{l} j^2 \cdot n_j \cdot d^2.  
\end{align*}
Here we have used the relation (\ref{deg})
for $C'=C_{j, k}$. 
Since $d>0$, the 
equality (\ref{eq:chi})
implies the desired equality for $\mathrm{wid}(C)$. 
\end{proof}

\appendix

\section{Uniqueness of Fourier-Mukai kernels}
Let $Y$ be a quasi-projective complex variety, or 
a spectrum of a completion of a finitely generated $\mathbb{C}$-algebra
at some maximum ideal. 
Suppose that 
$f_i \colon X_i \to Y$ are projective morphisms for $i=1, 2$, and 
$X_i$ are regular schemes. 
Given an object
\begin{align*}
\pP \in D^b \Coh(X_1 \times X_2)
\end{align*}
supported on $X_1 \times_Y X_2$, 
we have the 
Fourier-Mukai functor
\begin{align*}
\Phi_{X_1 \to X_2}^{\pP} \colon 
D^b \Coh(X_1) \to D^b \Coh(X_2)
\end{align*}
defined by 
\begin{align*}
\Phi_{X_1 \to X_2}^{\pP}
(-) \cneq 
\dR p_{2 \ast}(\dL p_1^{\ast}(-) \dotimes \pP)
\end{align*}
where $p_i \colon X_1 \times X_2 \to X_i$ is 
the projection. 
The above functor preserves
coherence
since $p_{2}|_{\Supp(\pP)}$ is projective. 
For another regular scheme 
$X_3$, a projective morphism
$f_3 \colon X_3 \to Y$ and 
an object $\qQ \in D^b \Coh(X_2 \times X_3)$
supported on $X_2 \times_Y X_3$, 
we have 
\begin{align*}
\Phi_{X_2\to X_3}^{\qQ} \circ \Phi_{X_1\to X_2}^{\pP}
\cong 
\Phi_{X_1 \to X_3}^{\qQ \circ \pP}
\end{align*}
where $\qQ \circ \pP$
is defined by
(cf.~\cite[Proposition~2.3]{Ch})
\begin{align*}
\qQ \circ \pP \cneq 
\dR p_{13\ast}(p_{12}^{\ast} \pP \dotimes p_{23}^{\ast} \qQ). 
\end{align*}
Here $p_{ij}
 \colon X_1 \times X_2 \times X_3 \to X_i \times X_j$ is the projection. 

If $Y=\Spec \mathbb{C}$ and 
$\Phi_{X_1 \to X_2}^{\pP}$ is an equivalence, 
then Orlov~\cite{Or1} showed that the kernel object $\pP$
is unique up to an isomorphism, i.e. 
$\Phi_{X_1 \to X_2}^{\pP} \cong \Phi_{X_1 \to X_2}^{\qQ}$
implies $\pP \cong \qQ$. 
It should be well-known that the same claim holds 
without $Y=\Spec \mathbb{C}$ assumption, but 
as the author cannot find a reference we include 
a proof here. 
\begin{lem}\label{lem:unique}
For $\pP, \qQ \in D^b \Coh(X_1 \times X_2)$
supported on $X_1 \times_Y X_2$, 
suppose that the following conditions hold: 
\begin{itemize}
\item We have an isomorphism of functors
$\Phi_{X_1 \to X_2}^{\pP} \cong \Phi_{X_1 \to X_2}^{\qQ}$. 
\item The functors $\Phi_{X_1 \to X_2}^{\pP}$, $\Phi_{X_1 \to X_2}^{\qQ}$
are equivalences. 
\end{itemize}
Then we have $\pP \cong \qQ$. 
\end{lem}
\begin{proof}
Let $\qQ^{\ast}$ be the object of $D^b \Coh(X_1 \times X_2)$
given by
\begin{align*}
\qQ^{\ast} \cneq \dR \hH om_{X_1 \times X_2}
(\qQ, \oO_{X_1 \times X_2}) \otimes 
p_1^{\ast}\omega_{X_1}[\dim X_1]. 
\end{align*}
By the Grothendieck duality, the functor 
$\Phi_{X_2 \to X_1}^{\qQ^{\ast}}$ is the right adjoint of 
$\Phi_{X_1 \to X_2}^{\qQ}$, hence an inverse of it. 
We have
\begin{align*}
 \Phi_{X_2 \to X_1}^{\qQ^{\ast}} \circ
\Phi_{X_1 \to X_2}^{\pP} \cong
\Phi_{X_1 \to X_1}^{\qQ^{\ast} \circ \pP}
\end{align*}
and it is isomorphic to the identity functor. 
Then $\Phi_{X_1 \to X_1}^{\qQ^{\ast} \circ \pP}$
sends $\oO_x$ to $\oO_x$ for any $x\in X_1$, 
and $\oO_{X_1}$ to $\oO_{X_1}$. 
Applying the
 argument of~\cite[Corollary~5.23]{Huybook}, it follows that
$\qQ^{\ast} \circ \pP \cong \oO_{\Delta_{X_1}}$.  
Similarly we have $\qQ \circ \qQ^{\ast} \cong \oO_{\Delta_{X_2}}$. 
We obtain
\begin{align*}
\pP \cong \oO_{\Delta_{X_2}} \circ \pP 
\cong \qQ \circ \qQ^{\ast} \circ \pP 
\cong \qQ \circ \oO_{\Delta_{X_1}} 
\cong \qQ
\end{align*}
as desired. 
\end{proof}

\section{Correction on flop-flop autoequivalence}
In this occasion, I would correct a wrong statement in~\cite[Section~3]{Tgen}
on the description of flop-flop autoequivalence. 
Let us consider the equivalence  
\begin{align}\label{FF:eq2}
\Phi_{X^{\dag} \to X}^{\oO_{X \times_Y X^{\dag}}}
\circ \Phi_{X \to X^{\dag}}^{\oO_{X \times_Y X^{\dag}}}
\colon D^b \Coh(X) \stackrel{\sim}{\to} D^b \Coh(X)
\end{align}
associated to the flop diagram (\ref{dia:flop}). 
In~\cite[Theorem~3.1]{Tgen}, 
it was stated that if 
$C$ is either a $(-1, -1)$-curve or a $(0, -2)$-curve, 
then the functor (\ref{FF:eq2}) is isomorphic 
to the generalized (commutative) twist functor $T_{\eE}$. 
However this
turns out to be wrong: the correct statement is that 
the equivalence (\ref{FF:eq2}) is 
an inverse of $T_{\eE}$. 
Indeed the statement in~\cite[Section~3]{Tgen}
that the equivalence 
\begin{align}\label{Feq2}
\Phi_{X \to X^{\dag}}^{\oO_{X \times_Y X^{\dag}}} \colon 
D^b \Coh(X) \stackrel{\sim}{\to} D^b \Coh(X^{\dag}). 
\end{align}
takes $\oO_{C}(-1)[1]$ to $\oO_{C^{\dag}}(-1)$
was wrong: 
it should be corrected that (\ref{Feq2}) takes 
$\oO_C(-1)$ to $\oO_{C^{\dag}}(-1)[1]$. 
Then replacing $T_{\eE}$ with $T_{\eE}^{-1}$
in the proof of~\cite[Theorem~3.1]{Tgen}, 
we obtain the statement that (\ref{FF:eq})
is isomorphic to $T_{\eE}^{-1}$. 

We explain why the above statement in~\cite[Section~3]{Tgen}
was wrong. 
In~\cite[Section~3]{Tgen}, 
I referred~\cite[Ver~1, Lemma~5.1]{Tst}, which 
in turn referred~\cite[(4.8)]{Br1}
that the equivalence (\ref{Feq2}) 
induces the equivalence
\begin{align}\label{ioPer}
\iPPer(X/Y) \stackrel{\sim}{\to} \oPPer(X^{\dag}/Y).
\end{align}
(Here we have used the fact that 
the equivalence (\ref{Feq2}) coincides with 
the equivalence $\Phi$ given in~\cite[Section~4]{Br1} by~\cite{Ch}). 
However (\ref{ioPer}) was not correct: 
it should 
be corrected that (\ref{Feq2}) 
induces the equivalence\footnote{In the notation of~\cite[(4.8)]{Br1}, 
the equivalence $\pPPer(W/X) \cong {\mathop{^{{p+1}}\rm{Per}}\nolimits}(Y/X)$
should be corrected as 
$\pPPer(W/X) \cong {\mathop{^{{p-1}}\rm{Per}}\nolimits}(Y/X)$}
\begin{align}\label{oiPer}
\oPPer(X/Y) \stackrel{\sim}{\to} 
\iPPer(X^{\dag}/Y).
\end{align}
Indeed let $\cC_X \subset \Coh(X)$
be the category of sheaves $F$ with $\dR f_{\ast}F=0$. 
Then~\cite[(4.5)]{Br1}
shows that 
(\ref{Feq2}) takes $\cC_X$ to $\cC_{X^{\dag}}[1]$. 
On the other hand, as $\pPPer(X/Y)$ is the gluing of 
$\Coh(Y)$ and $\cC_X[-p]$ (not $\cC_X[p]$) by the definition, 
the equivalence (\ref{Feq2}) should reduce the 
perversity one. 
After correcting (\ref{ioPer}) as (\ref{oiPer}), 
the argument of~\cite[Ver~1, Lemma~5.1]{Tst}
shows that (\ref{Feq2}) takes $\oO_{C}(-1)$ to $\oO_{C^{\dag}}(-1)[1]$.

\providecommand{\bysame}{\leavevmode\hbox to3em{\hrulefill}\thinspace}
\providecommand{\MR}{\relax\ifhmode\unskip\space\fi MR }
\providecommand{\MRhref}[2]{%
  \href{http://www.ams.org/mathscinet-getitem?mr=#1}{#2}
}
\providecommand{\href}[2]{#2}

Kavli Institute for the Physics and 
Mathematics of the Universe, University of Tokyo,
5-1-5 Kashiwanoha, Kashiwa, 277-8583, Japan.

\textit{E-mail address}: yukinobu.toda@ipmu.jp

\end{document}